\RequirePackage[l2tabu]{nag}
\documentclass[11pt,a4paper]{article}

\usepackage[T1]{fontenc}
\usepackage[utf8]{inputenc}
\usepackage{lmodern}

\usepackage{amssymb,amsthm}
\usepackage{mathtools}
\usepackage{stmaryrd}

\usepackage[hidelinks,pdftex,pdfusetitle]{hyperref}

\usepackage[vmargin=2cm,hmargin=3cm,includefoot]{geometry}

\usepackage{xcolor}
\hypersetup{
  colorlinks,
  linkcolor={red!50!black},
  citecolor={blue!50!black},
  urlcolor={blue!80!black}
}

\newcommand*{\N}{\mathbb{N}}
\newcommand*{\Z}{\mathbb{Z}}
\newcommand*{\R}{\mathbb{R}}

\newcommand*{\V}{\mathbb{V}}

\newcommand*{\std}{\mathrm{std}}
\newcommand*{\vol}{\mathrm{vol}}

\newcommand*{\bb}{\mathrm{b}}
\newcommand*{\p}{\mathrm{p}}
\newcommand*{\pb}{\mathrm{pb}}
\DeclareMathOperator{\reg}{reg}

\newcommand*{\calL}{\mathcal{L}}
\newcommand*{\calC}{\mathcal{C}}
\newcommand*{\calZ}{\mathcal{Z}}

\newcommand*{\frakB}{\mathfrak{B}}
\newcommand*{\frakX}{\mathfrak{X}}

\newcommand*{\rme}{\mathrm{e}}
\newcommand*{\rmf}{\mathrm{f}}

\newcommand*{\define}[1]{\textbf{#1}}

\DeclarePairedDelimiter\abs{\lvert}{\rvert}
\DeclarePairedDelimiter\norm{\lVert}{\rVert}
\DeclarePairedDelimiterX\innerp[2]{\langle}{\rangle}{#1,#2}
\DeclarePairedDelimiterX\intervCO[2]{[}{[}{#1,#2}
\DeclarePairedDelimiterX\intervOC[2]{]}{]}{#1,#2}
\DeclarePairedDelimiterX\intervCC[2]{[}{]}{#1,#2}
\DeclarePairedDelimiterX\intervOO[2]{]}{[}{#1,#2}
\DeclarePairedDelimiter\floor{\lfloor}{\rfloor}

\DeclareMathOperator{\id}{id}
\DeclareMathOperator{\pr}{pr}
\let\div\relax\DeclareMathOperator{\div}{div}
\DeclareMathOperator{\grad}{\nabla}
\DeclareMathOperator{\Hess}{Hess}
\DeclareMathOperator{\lapla}{\Delta}
\DeclareMathOperator{\levici}{\nabla^{\textsc{lc}}}

\DeclareMathOperator{\Reg}{Reg}

\DeclareMathOperator{\extdiff}{d}
\DeclareMathOperator{\Lie}{\calL}
\DeclareMathOperator{\trace}{tr}
\DeclareMathOperator{\interior}{int}
\DeclareMathOperator{\Ric}{Ric}
\newcommand{\extpow}[1]{\mathop{\bigwedge\nolimits^{#1}}}
\newcommand{\sympow}[1]{\mathop{\bigodot\nolimits^{#1}}}

\newtheorem{theorem}{Theorem}[section]
\newtheorem*{theorem*}{Theorem}

\newtheorem{lemma}[theorem]{Lemma}
\newtheorem{proposition}[theorem]{Proposition}
\theoremstyle{definition}
\newtheorem{definition}[theorem]{Definition}
\theoremstyle{remark}
\newtheorem{remark}[theorem]{Remark}

\newtheorem*{remarks*}{Remarks}
\newtheorem{example}[theorem]{Example}

\newcommand\blfootnote[1]{%
  \begingroup
  \renewcommand\thefootnote{}\footnote{#1}%
  \addtocounter{footnote}{-1}%
  \endgroup
}

\author{Beno\^{\i}t Jubin}
\title{Intrinsic volumes of sublevel sets}
\date{\today}

\begin{document}

\maketitle

\begin{abstract}
We
establish formulas that give the intrinsic volumes, or curvature measures, of sublevel sets of functions defined on Riemannian manifolds as integrals of functionals of the function and its derivatives.
For instance, in the Euclidean case, if $f \in \calC^3(\R^n, \R)$ and 0 is a regular value of $f$, then the intrinsic volume of degree $n-k$ of the sublevel set $M_f^0 \coloneqq f^{-1}(\intervOC{-\infty}{0})$, if the latter is compact, is given by
\begin{equation*}
\calL_{n-k}(M_f^0) =
\frac{\Gamma(k/2)}{2 \pi^{k/2} (k-1)!}
\int_{M^0} \div \left( \frac{P_{n, k}(\Hess(f), \grad f)}{\sqrt{f^{2(3k-2)} + \norm{\grad f}^{2(3k-2)}}} \grad f \right) \vol_n
\end{equation*}
for $1 \leq k \leq n$, where the $P_{n, k}$'s are polynomials given in the text.

This includes as special cases the Euler--Poincaré characteristic of sublevel sets and the nodal volumes of functions defined on Riemannian manifolds.
Therefore, these formulas give what can be seen as generalizations of the Kac--Rice formula.

Finally, we use these formulas to prove the Lipschitz continuity of the intrinsic volumes of sublevel sets.
\end{abstract}

\tableofcontents

\section*{Introduction}
\label{sec:intro}
\addcontentsline{toc}{section}{\nameref{sec:intro}}

Intrinsic volumes are geometric invariants attached to well-behaved subsets of Riemannian manifolds.
They include the volume and the Euler--Poincaré characteristic.
Among their applications in the field of integral geometry are Weyl's tube formula (\cite{weyl}), that gives the volumes of tubular neighborhoods of submanifolds, and the kinematic formula of Blaschke, Chern and Santal\'o (\cite{chern-kinematic, chern-kinematic2}), that gives the volume of the Minkowski sum of two convex bodies.
They were introduced in their modern form by Herbert Federer in the seminal article~\cite{federer}, where they are called curvature measures, after special cases in convex geometry were treated by Hermann Minkowski.
Among the vast literature on their subject, we only mention the book~\cite{morvan}, the survey on a related topic~\cite{thale}, and the articles~\cite{fu-chern, fu-wannerer, zahle-current}.

In this article, we study the intrinsic volumes of sublevel sets of functions defined on Riemannian manifolds.
These were already studied from the point of view of Morse theory in~\cite{fu}.
Since intrinsic volumes include the volume of the boundary, this study encompasses volumes of level sets, and in particular of zero sets, also called nodal sets.
The first closed explicit formulas computing nodal volumes appeared in~\cite{angst}, which was a motivation for the present article.
These formulas can be seen as generalizations of the so-called Kac--Rice formula (see for instance~\cite{nicolaescu}).

Sublevel sets are also studied in probability theory, where superlevel sets of random fields are called excursion sets; see for instance the books~\cite{adler, azais} and the articles~\cite{angst, lachieze, zahle}.
The importance of the formulas obtained in this paper for the study of random fields (as studied in~\cite{angst}), compared to existing Kac--Rice formulas, stems from the fact that they are in ``closed form'' as opposed to being limits of integrals depending on a parameter.

\paragraph{Main results}
We now describe the contents of this article in more detail.
In this introduction, we restrict ourselves to the flat case.
In Section~\ref{sec:volumes}, we recall the definition and main properties of intrinsic volumes.
If~$N$ is a flat compact $n$-dimensional Riemannian manifold with boundary, they take the form
\begin{equation}
\calL_{n-k}(N) =
b_k \int_{\partial N} \trace \left( \extpow{k-1} S \right) \vol_{\partial N}
\end{equation}
for $0 \leq k \leq n$, where $b_k \in \R$ and $S$ is the second fundamental form of~$\partial N$ in~$N$.

In Section~\ref{sec:sublevel}, we specialize our study to the case where $N$ is a sublevel set.
Namely, let~$M$ be a flat $n$-dimensional Riemannian manifold (without boundary), let $f \in \calC^3(M, \R)$, and assume that $a \in \R$ is a regular value of~$f$ and that the sublevel set $M^a \coloneqq f^{-1}(\intervOC{-\infty}{a})$ is compact.
The second fundamental form of~$\partial M^a$ in~$M$ can be expressed in terms of the gradient and the Hessian of~$f$.
An important lemma (Lemma~\ref{lem:poly}) establishes that the above integrand is then a polynomial in $\grad f$ and $\Hess(f)$ divided by $\norm{\grad f}^{2(k-1)}$.
We then use the divergence theorem to transform the above integral over $\partial M^a$ into an integral over $M^a$.
This leads to our main formula which, in the flat case, reads
\begin{equation}
\calL_{n-k}(M^a) =
b_k 
\int_{M^a} \div \left( \frac{P_{n, k}(\Hess(f), \grad f)}{\sqrt{(f-a)^{2(3k-2)} + \norm{\grad f}^{2(3k-2)}}} \grad f \right) \vol_M
\end{equation}
for $1 \leq k \leq n$, where the $P_{n, k}$'s are polynomials given in the text (Theorem~\ref{thm:main2}).
The main advantage of this formula is that it is an explicit integral over $M^a$ (and not $\partial M^a$) of a continuous functional in $f$ and its derivatives up to order~3.

Since the intrinsic volume of degree $n-1$ is half the volume of the boundary, this formula can be used to compute the volume of level sets.
If the ambient manifold $M$ is compact, one can use the intrinsic volume of either the sublevel or the superlevel set, yielding for the volume of the zero set $\calZ_f$ of $f$ the formula
\begin{equation}
\vol(\calZ_f) = \frac12 \int_M
\frac{\sigma_f}{\eta_f^3} \left(
f \norm{\grad f}^2 + \Hess(f)(\grad f, \grad f)
- \eta_f^2 \lapla f \right)
\vol_M
\end{equation}
where $\sigma_f$ is the sign of $f$ and $\eta_f \coloneqq \sqrt{f^2 + \norm{\grad f}^2}$.
This formula was obtained in the case of a flat torus in~\cite{angst}.
As in~\cite{angst}, one can do an integration by parts to eliminate $\sigma_f$ and obtain an integral over~$M$ where the integrand is a Lipschitz continuous functional of $f$ and its derivatives up to order~2 (Equation~\eqref{eq:lipschitz}).
This regularity allows one to apply techniques of the Malliavin calculus to obtain results about the expected value, variance and higher moments of the nodal volumes of certain families of random fields (see~\cite{angst}), and more generally of the intrinsic volumes of their excursion sets.

In Section~\ref{sec:continuity}, after recalling basic facts on natural topologies on $\calC^p(M, \R)$, the uniform and the (Whitney) strong $\calC^p$-topologies, we prove that conditions needed to establish our formulas (regularity of the value and compactness of the sublevel sets) are generic.
Then, we prove the continuity of intrinsic volumes of sublevel sets.
For instance, if $0 \leq k \leq n$, then the function
\begin{align}
\calL_{n-k}^\mathrm{sub} \colon \Reg_\pb^3(M, \R)_U &\longrightarrow \R\\
(f, a) & \longmapsto \calL_{n-k}(M^a_f)\notag
\end{align}
is Lipschitz continuous, where the domain is the set of couples $(f, a)$ where $f \in \calC^3(M, \R)$ is proper bounded below and $a \in \R$ is a regular value of $f$, and is equipped with the uniform $\calC^3$-topology (Theorem~\ref{thm:continuity}).
In particular, the Euler--Poincaré characteristic of sublevel sets is locally constant.

\paragraph{Conventions and notation}

\begin{itemize}
\item
If $P$ is a proposition, then $[P] \coloneqq 1$ if $P$ else 0.
\item
We denote by $\pr_i$ the projection on the $i\textsuperscript{th}$ factor of a direct product.
\item
The bracket $\floor{-} \colon \R \to \Z$ denotes the floor function.
\item
If $a, b \in \R$, then $\llbracket a, b \rrbracket \coloneqq \Z \cap \intervCC{a}{b}$.
\item
If $a \in \R$, then $\overline{\N_{\geq a}} \coloneqq \{ n \in \N \mid n \geq a \} \cup \{\infty\}$ and similarly for similar symbols.
\item
The symbol $\bigodot$ (resp.\ $\bigwedge$) denotes the symmetric (resp.\ exterior) product or power of vector spaces.
\item
Unless otherwise specified, manifolds are Hausdorff, paracompact, real, finite-dimensional, and smooth, that is, of class $\calC^\infty$.
\item
The space of smooth sections of the vector bundle $E \to M$ is denoted by $\Gamma(E \to M)$ or $\Gamma^{(p)}(E \to M)$ if the differentiability class $p \in \overline{\N}$ need be specified.
For instance, the metric tensor of a Riemannian manifold~$M$ is an element of $\Gamma(\sympow{2}T^*M \to M)$.
\end{itemize}

\section{Intrinsic volumes}
\label{sec:volumes}

Let~$(M, g)$ be an $n$-dimensional compact Riemannian manifold with boundary.
Its metric will also be denoted by $\innerp{-}{-}$ and the associated norm by $\norm{-}$.
We denote by $\levici$ its Levi-Civita connection.
Let $\vol_M$ be the Riemannian density on $M$ and $\vol_{\partial M}$ be the induced density on $\partial M$.
The symbol $\vol$ will also denote the associated volume of (sub)manifolds.

Let~$R \in \Gamma(\sympow{2} \extpow{2}T^*M \to M)$ be the covariant curvature tensor of~$M$.
Let $S \coloneqq ((\levici \nu|_{T\partial M})^T)^\flat \in \Gamma(\sympow{2} T^*\partial M \to \partial M)$ be the second fundamental form of~$\partial M$ in~$M$, where $\nu \in \Gamma(TM|_{\partial M} \to \partial M)$ is the outward unit normal vectorfield on~$\partial M$ and $(-)^T \colon TM|_{\partial M} \to T\partial M$ denotes the tangential component, and $\flat \colon T\partial M \to T^*\partial M$ denotes the musical isomorphism induced by the metric.
The symbol $\trace$ denotes the trace of a bilinear form.

For the exterior product of symmetric bilinear forms, $\wedge \colon \sympow{2}\extpow{p}\V \times \sympow{2}\extpow{q}\V \to \sympow{2}\extpow{p+q}\V$, also called in differential geometry the Kulkarni--Nomizu product, see for instance~\cite[§2]{federer}.

For $0 \leq k \leq n$, the \define{intrinsic volume of degree~$n-k$ of~$M$} is defined as
\begin{multline}
\label{eq:volumes}
\calL_{n-k}(M) \coloneqq
a_k \int_{M} \trace\left( \extpow{k/2} R \right) \vol_M + {}\\
\sum_{m=0}^{\floor*{\frac{k-1}{2}}} b_{k, m} \int_{\partial M} \trace \left( \extpow{m} R|_{\partial M} \wedge \extpow{k-1-2m} S \right) \vol_{\partial M}
\end{multline}
where
\begin{align}
a_k &\coloneqq \frac{[k \text{ even}]}{(-2\pi)^{k/2} \: (k/2)!} &\text{for } 0 \leq k,\\
b_{k, m} &\coloneqq \frac{(-1)^m \Gamma(k/2 - m)}{2^{m+1} \pi^{k/2} m! (k-1-2m)!} &\text{for } 0 \leq m \leq \floor*{\frac{k-1}2}.
\intertext{We also set}
b_k &\coloneqq b_{k, 0} = \frac{\Gamma(k/2)}{2 \pi^{k/2} (k-1)!} &\text{for } 1 \leq k.
\end{align}

\begin{remark}
Formula~\eqref{eq:volumes} was obtained by specializing the general definition~\cite[Def.~10.7.2]{adler}, which holds for Whitney stratified spaces ``of positive reach'' in Riemannian manifolds (including Riemannian manifolds with corners), to the case of Riemannian manifolds with boundary.
The integrands, which are contractions of the curvature tensor and the second fundamental form, are called the \define{Lipschitz--Killing curvatures} of~$\partial M$ in~$M$.
More general versions $\calL_{n-k}(M, A)$ can be defined for Borel subsets $A \subseteq M$ and are called \define{curvature measures} in~$M$.
The intrinsic volumes are the total measures of these curvature measures, that is, $\calL_{n-k}(M) = \calL_{n-k}(M, M)$.
\end{remark}

One has
\begin{align}
\calL_n(M) &= \vol(M),\\
\label{eq:vol-boundary}
\calL_{n-1}(M) &= \frac12 \vol(\partial M),\\
\calL_0(M) &= \chi(M).
\end{align}
The first two equalities immediately follow from $a_0 = 1$ and from $a_1 = 0$ and $b_1 = \frac12$ respectively (indeed, $\trace(\extpow{0}R_x)$ is the trace of the identity of $\extpow{0}T^*_xM \simeq \R$, which is~1, and similarly for $\trace(\extpow{0}S_x)$).
The third equality is the Gauss--Bonnet--Chern theorem (see~\cite{chern, chern-curvatura} for the original articles, and~\cite{spivak} for manifolds with boundary), where the right-hand side is the Euler--Poincaré characteristic of~$M$, and in particular is an integer and is zero in the odd-dimensional boundaryless case.

Since $a_2 = -b_2 = (2\pi)^{-1}$, one has
\begin{equation}
\calL_{n-2}(M) =
- \frac1{2\pi} \int_{M} \operatorname{scal} \: \vol_M + \frac1{2\pi} \int_{\partial M} (\trace S) \: \vol_{\partial M}
\end{equation}
where $\operatorname{scal}$ denotes the scalar curvature of~$M$.
Note that $\trace S$ is $(n-1)$~times the mean curvature of~$\partial M$ in~$M$.

If $M$ is flat, that is, $R = 0$, then Formula~\eqref{eq:volumes} simplifies, since only the summand corresponding to $m = 0$ may be nonzero, giving
\begin{equation}
\label{eq:vol-flat}
\calL_{n-k}(M) =
b_k \int_{\partial M} \trace \left( \extpow{k-1} S \right) \vol_{\partial M}
\end{equation}
for $1 \leq k \leq n$.
In particular,
\begin{align}
\calL_{n-2}(M) &=
\frac1{2\pi} \int_{\partial M} (\trace S) \: \vol_{\partial M},\qquad\text{ and}\\
\chi(M) &=
b_n \int_{\partial M} (\det S) \: \vol_{\partial M}\quad\text{if $n \geq 1$}.
\end{align}
Note that $\trace S$ is $(n-1)$~times the mean curvature, and $\det S$ ``the'' Lipschitz--Killing, or Gauss--Kronecker, curvature, of~$\partial M$ in~$M$.

\section{Intrinsic volumes of sublevel sets}
\label{sec:sublevel}

\subsection{Sublevel sets and level sets}
\label{subsec:sublevel}

If $f \colon M \to \R$ is a function on a set and $a \in \R$, then the \define{$a$-sublevel set} of $f$ is defined by
\begin{equation}
M^a_f \coloneqq f^{-1} \big( \intervOC{-\infty}{a} \big)
\end{equation}
also written $M^a$ if there is no risk of confusion, and the \define{$a$-level set} of $f$ is $f^{-1}(a)$.

Let $M$ be a manifold, let $p \in \overline{\N_{\geq 1}}$, and let $f \in \calC^p(M, \R)$.
The real number $a \in \R$ is a \define{regular value} of $f$ if $f(x) = a$ implies $\extdiff f(x) \neq 0$ for all $x \in M$.
We define the sets
\begin{align}
\label{def:reg}
\Reg^p(M, \R) &\coloneqq \{ (f, a) \in \calC^p(M, \R) \times \R \mid a \text{ is a regular value of $f$} \},\\
\label{def:regc}
\Reg_c^p(M, \R) &\coloneqq \{ (f, a) \in \Reg^p(M, \R) \mid M_f^a \text{ is compact}\}.
\end{align}
We also set $\calC^p_{a-\reg}(M, \R) \coloneqq \pr_1 \big( \Reg^p(M, \R) \cap (\calC^p(M, \R) \times \{a\}) \big)$ and similarly for $\calC^p_{a-\reg, c}(M, \R)$.

\begin{proposition}
\label{prop:levsublev}
Let $M$ be a manifold, let $p \in \overline{\N_{\geq 1}}$, and let $(f, a) \in \Reg^p(M, \R)$.
Then, $M^a_f$ is a full-dimensional $\calC^p$-submanifold with boundary of~$M$.
Its manifold boundary, equal to its topological boundary, is the $\calC^p$-hypersurface $\partial M^a_f = f^{-1}(a)$.
\end{proposition}

\begin{proof}
By the submersion theorem, $f^{-1}(a)$ is a $\calC^p$-hypersurface of~$M$.
By considering separately points $x \in M$ such that $f(x) < a$ and such that $f(x) = a$, one checks that $M^a_f$ is a full-dimensional $\calC^p$-submanifold with boundary of~$M$, and that $\partial M^a_f = f^{-1}(a)$.
\end{proof}

\subsection{Intrinsic volumes of sublevel sets}
\label{subsec:volume-sublevel}

Let $(M, g)$ be an $n$-dimensional Riemannian manifold (not necessarily compact, but without boundary).
Let $f \in \calC^2(M, \R)$.
Its \define{gradient} is defined by $\grad f \coloneqq (\extdiff f)^\sharp$.
Its \define{Hessian} is defined by $\Hess(f) \coloneqq \levici \extdiff f = (\levici \grad f)^\flat$.
Its \define{Laplacian} is the trace of its Hessian, $\lapla f \coloneqq \trace (\Hess(f))$.

Let $(f, a) \in \Reg^2(M, \R)$.
By Proposition~\ref{prop:levsublev}, the set $M^a$ is a full-dimensional $\calC^2$-submanifold with boundary of~$M$ and its boundary is the $\calC^2$-hypersurface $\partial M^a$.
The outward unit normal vectorfield of $\partial M^a$ is $\nu = \frac{\grad f}{\norm{\grad f}}$.
Therefore, one has $\levici \nu = \norm{\grad f}^{-1} \levici \grad f + \extdiff \norm{\grad f}^{-1} \odot \grad f$, which has tangential component $\norm{\grad f}^{-1} \levici \grad f$.
Therefore, the second fundamental form of~$\partial M^a$ in~$M$ is given by
\begin{equation}
\label{eq:hessian}
S = \frac{\Hess(f)|_{\grad f^\perp}}{\norm{\grad f}}.
\end{equation}

We briefly explain the idea underlying the rest of this subsection.
By Formula~\eqref{eq:hessian}, the integrals in the sum on~$m$ in Formula~\eqref{eq:volumes} are equal to
\begin{equation}
\int_{\partial M^a} \norm{\grad f}^{2m+1-k} \trace \left( \extpow{m} R|_{\grad f^\perp} \wedge \extpow{k-1-2m} \Hess(f)|_{\grad f^\perp} \right) \vol_{\partial M^a}.
\end{equation}
We will convert these integrals on $\partial M^a$ into integrals on $M^a$ by using the divergence theorem.
To do this, we need to find a vectorfield $X \in \frakX(M^a)$ such that $X|_{\partial M^a} = \nu = \frac{\grad f}{\norm{\grad f}}$ and $\norm{\grad f}^{2m+1-k} \trace \left( \extpow{m} R|_{\grad f^\perp} \wedge \extpow{k-1-2m} \Hess(f)|_{\grad f^\perp} \right) X$ has a divergence which is integrable on~$M^a$.
Besides the boundary, the possibly problematic points are the points where $\grad f = 0$, first because of the factor $\norm{\grad f}^{2m+1-k}$, and also because of the restriction to $\grad f^\perp$.
Since the two regions of interest are at $f = a$ and at $\grad f = 0$, it makes sense to look for a vectorfield of the form $X = F \circ \frac{\grad f}{a-f}$ where $F \in \calC^1(TM, TM)$ is such that $F(u) \sim_\infty \frac{u}{\norm{u}}$ and $F$ vanishes sufficiently fast at~0 for the divergence to be integrable.

We now make this idea precise.

\begin{lemma}
\label{lem:poly}
For all $n, k, m \in \N$ such that $1 \leq k \leq n$ and $0 \leq m \leq \floor{\frac{k-1}2}$, there exists a homogeneous polynomial $P_{n, k, m}$ with integer coefficients such that for any $n$-dimensional Euclidean space with orthonormal basis $(\V, \frakB)$, any symmetric bilinear forms $R \in \sympow{2}\extpow{2}\V^*$ and $H \in \sympow{2}\V^*$, and any $v \in \V \setminus \{0\}$, one has
\begin{equation}
\label{eq:lem}
\trace \left( \extpow{m} R|_{v^\perp} \wedge \extpow{k-1-2m} H|_{v^\perp} \right) = \frac{P_{n, k, m}\big((r_{ijkl}), (h_{ij}), (v_i)\big)}{\norm{v}^{2(k-1)}}
\end{equation}
where $(r_{ijkl})$ (resp.\ $(h_{ij})$ and $(v_i)$) are the coefficients of~$R$ (resp.\ $H$ and $v$) in~$\frakB$.
\end{lemma}

For the sake of definiteness, if $\frakB = (e_i)_{1 \leq i \leq n}$ is an orthonormal basis of $\V$, we consider the basis $(e_i \wedge e_j)_{1 \leq i < j \leq n}$ of $\extpow{2}\V$.
The coefficients of $R$ can be written $(r_{ijkl})_{1 \leq i, j, k, l \leq n}$ with $i < j$ and $k < l$ and $(i, j) \leq (k, l)$ in the lexicographic order.

\begin{proof}
Let $n, k, m, (\V, \frakB), R, H, v$ be as in the statement.
Without loss of generality, we can suppose that $(\V, \frakB) = (\R^n, \std)$ with the standard inner product.
Set $a_i \coloneqq \sqrt{\sum_{j=1}^i v_i^2}$ for $1 \leq i \leq n$.
In particular, $a_1 = v_1$ and $a_n = \norm{v}$.
We first assume that $v_1 > 0$.
Let $P$ be the following change of basis matrix:
$P_{i1} \coloneqq v_i/\norm{v}$, and if $j \geq 2$, then $P_{ij} \coloneqq \beta_{ij}/(a_{j-1} a_j)$ with $\beta_{ij} \coloneqq -v_i v_j$ if $i < j$, or $a_{j-1}^2$ if $i = j$, or 0 if $j < i$.
It is an orthogonal matrix and $P^{-1}HP$ restricted to the rows and columns $2 \leq i, j \leq n$ is the matrix of $H|_{v^\perp}$ in an orthonormal basis.

Since the $\beta_{ij}$'s are polynomials in the $v_k$'s and $a_k$'s, the coefficients $(P^{-1}HP)_{ij}$ with $2 \leq i, j \leq n$ are of the form
\begin{equation*}
(P^{-1}HP)_{ij} = \frac{\big(\text{ polynomial } \big)}{a_{i-1} a_i a_{j-1} a_j}.
\end{equation*}

The change of basis matrix in $\extpow{2} \V$ associated with $P$, say $Q$, has coefficients $Q_{ijkl} = P_{ik}P_{jl} - P_{il}P_{jk}$.
As with~$P$, the matrix $Q^{-1}RQ$ restricted to the rows and columns $2 \leq i, j, k, l \leq n$ (with $i < j$ and $k < l$) is the matrix of $R|_{v^\perp}$ in an orthonormal basis.
The coefficients $(Q^{-1}RQ)_{ijkl}$ with $2 \leq i, j, k, l \leq n$ (with $i < j$ and $k < l$) are of the form
\begin{equation*}
(Q^{-1}RQ)_{ijkl} = \frac{\big(\text{ polynomial } \big)}{a_{i-1} a_i a_{j-1} a_j a_{k-1} a_k a_{l-1} a_l}.
\end{equation*}

The coefficients of the exterior product of their exterior powers is again of a similar form, hence so is its trace.
More precisely, it is a rational fraction with variables $r_{ijkl}$, $h_{ij}$, $v_i$, $a_i$.
The denominator is a product of $a_i$'s, where the exponent of $a_n$ is at most $2(2m + (k-1-2m)) = 2(k-1)$.

This expression was obtained under the assumption that $v_1 > 0$, but it is intrinsic to $(R, S, v)$ and invariant under orthogonal transformations of $\V$.
Therefore, it also holds if $v = \rme_n$, in which case all the $a_i$'s with $i < n$ vanish.
As a consequence, the only $a_i$'s at the denominator are those with $i = n$, that is, $a_n = \norm{v}$.

The variable $a_n$ does not appear in the numerator (since $\beta_{ij}$ only involves $a_{j-1}$).
For $i < n$, then $a_i$ as a function of the $v_k$'s is not differentiable at $\rme_n$ but is differentiable at~$\rme_1$, so by invariance under orthogonal transformation, the variables $a_i$ with $i < n$ can only appear in the numerator with even exponents.
Therefore, the numerator is a polynomial in the coefficients $r_{ijkl}$, $h_{ij}$, $v_i$.
\end{proof}

\begin{definition}
For all $n, k, m \in \N$ such that $1 \leq k \leq n$ and $0 \leq m \leq \floor{\frac{k-1}2}$, we define $P_{n, k, m}$ to be the (unique) polynomial whose existence is asserted in Lemma~\ref{lem:poly}.
For other values of the indices, we set $P_{n, k, m} \coloneqq 0$.
We set $P_{n, k} \coloneqq P_{n, k, 0}$.
\end{definition}

\begin{remark}
\label{rmk:degree}
One has $\dim \sympow{2}\V^* = \frac{n(n+1)}2$ and $\dim \sympow{2}\extpow{2}\V^* = \frac{n(n-1)(n(n-1)+2)}8$.
Therefore, $P_{n, k, m}$ has $n + [k-1-2m \geq 1] \frac{n(n+1)}2 + [m \geq 1] \frac{n(n-1)(n(n-1)+2)}8$ variables.
By homogeneity considerations, $P_{n, k, m}$ has degree $2(k-1)$ in the coefficients of $v$, degree $k-1-2m$ in the coefficients of $H$, and degree $m$ in the coefficients of $R$.
\end{remark}

The proof shows how to compute the $P_{n, k, m}$'s.
For instance, one has
\begin{multline*}
(P^{-1}HP)_{ij} = \frac1{a_{i-1} a_i a_{j-1} a_j} \times {}\\
\left( a_{i-1}^2 h_{ij} a_{j-1}^2 + v_i v_j \sum_{k=1}^{i-1} \sum_{l=1}^{j-1} v_k h_{kl} v_l - a_{i-1}^2 v_j \sum_{l=1}^{j-1} h_{il} v_l - a_{j-1}^2 v_i \sum_{k=1}^{i-1} v_k h_{kj} \right)
\end{multline*}
for $2 \leq i, j \leq n$, and the trace of an exterior power can be computed as a sum of minors of given order.
We consider a few special cases:
\begin{itemize}
\item
If $k = 1$ (hence $m = 0$), then the left-hand side of~\eqref{eq:lem} is the trace of the identity on $\extpow{0} v^\perp \simeq \R$, so $P_{n, 1} = 1$.
\item
For the case $k = 2$ (hence $m = 0$), note that $\trace(H|_{v^\perp}) = \trace H - \norm{v}^{-2} \: H(v, v)$.
This gives
\begin{equation}
P_{n, 2} = \sum_{i=1}^n \left( \sum_{j \neq i} v_j^2 \right) h_{ii} - 2 \sum_{i<j} v_i v_j h_{ij}.
\end{equation}
The first such polynomial is $P_{2, 2} = v_2^2 h_{11} - 2 v_1 v_2 h_{12} + v_1^2 h_{22}$.
This is simply $H(u, u)$ where $u$ is any of the two unit vectors orthogonal to $v$.
\item
For the case $(k, m) = (n, 0)$, one has $\det H = \det(H|_{v^\perp}) \: \norm{v}^{-2} \: H(v, v^H)$ where $v^H$ is the projection of $v$ onto $(v^\perp)^{\perp_H}$ parallel to $v^\perp$, provided that $H|_{v^\perp}$ is nondegenerate.
Indeed, considering an orthonormal basis $(\rmf_i)_{1 \leq i \leq n}$ of $\V$ with $\rmf_1 = \frac{v}{\norm{v}}$ and $(\rmf_i)_{2 \leq i \leq n}$ a diagonalizing basis for $H|_{v^\perp}$, one has
\begin{equation*}
\det H = \det
\begin{pmatrix}
H(\frac{v}{\norm{v}}, \frac{v}{\norm{v}}) & \dots & H(\frac{v}{\norm{v}}, \rmf_i) & \dots\\
\vdots & \ddots & 0 & 0\\
H(\frac{v}{\norm{v}}, \rmf_i) & 0 & \lambda_i & 0\\
\vdots & 0 & 0 & \ddots
\end{pmatrix}.
\end{equation*}
A double expansion of this determinant yields
\begin{equation*}
\det H = (\det H|_{v^\perp}) \: \norm{v}^{-2} \: H\left(v, v - \sum_{i=2}^n H(v, \frac{\rmf_i}{\lambda_i}) \rmf_i\right)
\end{equation*}
as claimed.
\end{itemize}

Applying Lemma~\ref{lem:poly} to $(R, S, v) = (R, \Hess(f), \grad f)$, one obtains
\begin{equation}
\label{eq:integrand}
\trace \left( \extpow{m} R|_{\grad f^\perp} \wedge \extpow{k-1-2m} \Hess(f)|_{\grad f^\perp} \right) =
\frac{P_{n, k, m}(R, \Hess(f), \grad f)}{\norm{\grad f}^{2(k-1)}}.
\end{equation}
In view of the special cases considered above, and under the nondegeneracy condition for the third equation, one has
\begin{align}
P_{n, 1}(\Hess(f), \grad f) &= 1,\\
P_{n, 2}(\Hess(f), \grad f) &= \norm{\grad f}^2 \lapla f - \Hess(f)(\grad f, \grad f),\\
P_{n, n}(\Hess(f), \grad f) &= \frac{\det(\Hess(f)) \: \norm{\grad f}^{2n}}{\Hess(f)\left(\grad f, (\grad f)^{\Hess(f)}\right)}.
\end{align}

We now state the version of the divergence theorem that will be useful to us.
The \define{divergence} of a $\calC^1$-vectorfield $X \in \frakX(M)$ is defined by $\Lie_X \vol_M = (\div X) \: \vol_M$, where $\Lie$ denotes the Lie derivative.
Many generalizations of the standard divergence theorem have been proved, relaxing hypotheses on the regularity and compactness of the manifold or stratified space and on the regularity of the vectorfield, encompassing the present statement.
We include a proof for the convenience of the reader.

\begin{theorem}[Divergence theorem]
\label{thm:div}
Let $(M, g)$ be a compact $n$-dimensional Riemannian manifold with boundary.
Let $X \in \frakX(M)$ be a continuous vectorfield on $M$ which is of class $\calC^1$ on $\interior{M}$ and such that $\div X \in L^1(M)$.
Then,
\begin{equation}
\int_M (\div X) \vol_M = \int_{\partial M} \innerp*{X}{v} \vol_{\partial M}.
\end{equation}
\end{theorem}

\begin{proof}
If $X$ is of class $\calC^1$ on $M$, then this is the standard divergence theorem.
Else, we consider the geodesic flow from the boundary of $M$ along the outward unit normal vectorfield~$\nu$.
For $\epsilon > 0$ small enough, set $\theta_\epsilon \colon \partial M \to M, x \mapsto \exp(x,-\epsilon \nu_x)$ and set $M_\epsilon \coloneqq M \setminus \bigcup_{s \in \intervCO{0}{\epsilon}} \theta_s(\partial M)$.
For $\epsilon$ small enough, $M_\epsilon$ is a compact submanifold with boundary of $M$, and $\theta_\epsilon$ induces a diffeomorphism $\vartheta_\epsilon \colon \partial M \xrightarrow{\sim} \partial M_\epsilon$.
Applying the standard divergence theorem on $M_\epsilon$, one obtains
$\int_{M_\epsilon} (\div X) \: \vol_M = \int_{\partial M_\epsilon} \innerp*{X}{v} \: \vol_{\partial M_\epsilon}$.
When $\epsilon \to 0$, the left-hand side converges to $\int_{M_\epsilon} (\div X) \: \vol_M$ by Lebesgue's dominated convergence theorem, since $\div X \in L^1(M)$.
The right-hand side is equal, by change of variable, to $\int_{\partial M} \innerp*{\vartheta_\epsilon^* X}{v} (\det T\vartheta_\epsilon) \: \vol_{\partial M}$, which converges to $\int_{\partial M} \innerp*{X}{v} \: \vol_{\partial M}$ since the integrand is uniformly convergent and $\partial M$ is compact.
\end{proof}

We can now prove a first general result.

\begin{theorem}
\label{thm:main}
Let $(M, g)$ be an $n$-dimensional Riemannian manifold.
Let $(f, a) \in \Reg_c^3(M, \R)$.
For $1 \leq k \leq n$ and $0 \leq m \leq \floor{\frac{k-1}2}$, let $F_{k, m} \in \calC^1(TM, TM)$ be such that $F_{k, m}(u) \sim_\infty \frac{u}{\norm{u}}$.
Then, for $0 \leq k \leq n$, one has
\begin{multline}
\calL_{n-k}(M^a) = a_k
\int_{M^a} \trace\left( \extpow{k/2} R \right) \vol_M + {}\\
\sum_{m=0}^{\floor*{\frac{k-1}{2}}} b_{k, m} 
\int_{M^a} \div \left( \norm{\grad f}^{2m+3(1-k)}
P_{n, k, m}(R, \Hess(f), \grad f)
\left( F_{k, m} \circ \frac{\grad f}{a-f} \right) \right) \vol_M
\end{multline}
under the condition that the divergence appearing in the integral exists and is integrable.

If $M$ is flat, then for $1 \leq k \leq n$, let $F_k \in \calC^1(TM, TM)$ be such that $F_k(u) \sim_\infty \frac{u}{\norm{u}}$.
Then,
\begin{equation}
\calL_{n-k}(M^a) =
b_k \int_{M^a} \div \left( \norm{\grad f}^{3(1-k)} P_{n, k}(\Hess(f), \grad f) \left( F_k \circ \frac{\grad f}{a-f} \right) \right) \vol_M.
\end{equation}
under the same conditions.
\end{theorem}

\begin{remark}
By ``$F(u) \sim_\infty \frac{u}{\norm{u}}$'', we mean that $\lim_{\norm{u} \to +\infty} d \left( F(u), \frac{u}{\norm{u}} \right) = 0$, where $d$ is the distance on $TM$ induced by the Riemannian metric of $M$ (or any distance, since $M^a$ is compact and $\frac{u}{\norm u}$ has unit norm).
\end{remark}

\begin{proof}
Starting with the definition~\eqref{eq:volumes}, we use the expression of the second fundamental form~\eqref{eq:hessian} and Equation~\eqref{eq:integrand} to obtain
\begin{multline}
\calL_{n-k}(M) \coloneqq
a_k \int_{M} \trace\left( \extpow{k/2} R \right) \vol_M + {}\\
\sum_{m=0}^{\floor*{\frac{k-1}{2}}} b_{k, m} \int_{\partial M} \frac{P_{n, k, m}(R, \Hess(f), \grad f)}{\norm{\grad f}^{3(k-1)-2m}}
\vol_{\partial M}.
\end{multline}
The asymptotic property of $F_{k, m}$ ensures that the vectorfield whose divergence is considered in the statement is continuous on $\partial M^a$ and its value there is
\begin{equation*}
\norm{\grad f}^{2m+3(1-k)} P_{n, k, m}(R, \Hess(f), \grad f) \frac{\grad f}{\norm{\grad f}}.
\end{equation*}
Finally, the hypotheses of the proposition ensure that the divergence theorem applies.
\end{proof}

\begin{remark}
\label{rmk:C2}
Since $P_{n, 1} = 1$, the theorem for $k = 1$ holds for $(f, a) \in \Reg^2_c(M, \R)$.
\end{remark}

Our next step is to find explicit functions $F$ (in particular proving that some exist) making the divergence appearing in the theorem integrable.
We consider radial maps of the form $F_{k, m}(u) = \norm{u}^{3(k-1)-2m} G_{k, m}(\norm{u}) u$ with $G_{k, m} \in \calC^1(\R_{\geq 0}, \R)$.
The condition $G_{k, m}(x) \sim_{+\infty} x^{2(m+1)-3k}$ ensures that $F_{k, m}(u) \sim_\infty \frac{u}{\norm{u}}$.
Examples of functions $G$ satisfying these conditions are given by $G_{k, m}(x) \coloneqq \left( 1 + x^{2(3k-2(m+1))}\right)^{-1/2}$.
We set
\begin{equation}
\label{eq:etaell}
\eta_{f, \ell} \coloneqq \sqrt{f^{2\ell} + \norm{\grad f}^{2\ell}}
\end{equation}
for $\ell \geq 0$.
These choices for $F$ yield the following theorem.

\begin{theorem}
\label{thm:main2}
Let $(M, g)$ be an $n$-dimensional Riemannian manifold.
Let $(f, a) \in \Reg_c^3(M, \R)$.
For $0 \leq k \leq n$, one has
\begin{multline}
\label{eq:main}
\calL_{n-k}(M^a) = a_k
\int_{M^a} \trace\left( \extpow{k/2} R \right) \vol_M + {}\\
\sum_{m=0}^{\floor*{\frac{k-1}{2}}} b_{k, m} 
\int_{M^a} \div \left(
\frac{P_{n, k, m}(R, \Hess(f), \grad f)}{\eta_{f-a, 3k-2(m+1)}} \grad f \right) \vol_M.
\end{multline}

If $M$ is flat and $1 \leq k \leq n$, then
\begin{equation}
\label{eq:main-flat}
\calL_{n-k}(M^a) =
b_k 
\int_{M^a} \div \left( \frac{P_{n, k}(\Hess(f), \grad f)}{\eta_{f-a, 3k-2}} \grad f \right) \vol_M.
\end{equation}
\end{theorem}

For $k = 1, 2$, this gives
\begin{align}
\label{eq:nodal}
\calL_{n-1}(M^a) &= \frac12 \int_{M^a} \div \left( \frac{\grad f}{\eta_{f-a}} \right) \vol_M, \\
\calL_{n-2}(M^a) &= \begin{multlined}[t]
-  \frac1{2\pi} \int_{M^a} \operatorname{scal} \: \vol_M + {}\\
\frac1{2\pi} \int_{M^a} \div \left( \frac{\norm{\grad f}^2 \lapla f - \Hess(f)(\grad f, \grad f)}{\eta_{f-a, 4}} \grad f \right) \vol_M.
\end{multlined}
\end{align}
Similarly, when $M$ is flat and $\Hess(f)|_{\grad f^\perp}$ is nondegenerate, if $n \geq 1$, one has
\begin{equation}
\chi(M^a) = b_n
\int_{M^a} \div \left(
\frac{\det(\Hess(f)) \: \norm{\grad f}^{2n}}{\Hess(f)\left(\grad f, (\grad f)^{\Hess(f)}\right) \eta_{f-a, 3n-2}} \grad f \right) \vol_M.
\end{equation}

\begin{remark}
There are obviously many natural choices for the functions $F$ and $G$.
For instance, one can take $F_{k, m} \coloneqq F_k$.
With the $F_k$'s given above, the divergence corresponding to the $m\textsuperscript{th}$ summand reads $\div \left( \norm{\grad f}^{2m} \frac{P_{n, k, m}(R, \Hess(f), \grad f)}{\eta_{f-a, 3k-2}} \grad f \right)$.
In the case of nodal volumes, other choices are given in the next subsection.
\end{remark}

\begin{remark}
\label{rmk:corners}
Intrinsic volumes can be defined for Riemannian manifolds with corners, and even Whitney stratified spaces of ``positive reach'' in Riemannian manifolds.
Since the divergence theorem admits generalizations to these settings, it is possible to extend the above results to sublevel sets of functions defined on Riemannian manifolds with boundary or corners, and to Whitney stratified spaces in Riemannian manifolds, under the assumption that the function is transverse to the boundary or the strata respectively.
Boundary terms will appear in the formulas.
We do not carry out this generalization in full and only give a formula for nodal volumes in the next subsection (see Remark~\ref{rmk:corners2}).
\end{remark}

\subsection{Nodal volumes}
\label{subsec:nodal-volumes}

In this subsection, we show how we can compute the intrinsic volumes of the zero sets, or nodal sets, of functions defined on compact Riemannian manifolds.
Let $(M, g)$ be a compact $n$-dimensional Riemannian manifold.
Let $f \in \calC^2_{0-\reg}(M, \R)$ (class $\calC^2$ is sufficient by Remark~\ref{rmk:C2}).
The \define{zero set} of $f$ is $\calZ_f \coloneqq f^{-1}(0) = \partial M^0_f = \partial M^0_{-f}$.
By~\eqref{eq:vol-boundary}, one has $\vol(\calZ_f) = \calL_{n-1}(M^0_f) + \calL_{n-1}(M^0_{-f})$.
Since $M^0_f \cup M^0_{-f} = M$ and $M^0_f \cap M^0_{-f} = \calZ_f$ is negligible in $M$, Formula~\eqref{eq:nodal} gives an integral on~$M$.
Using the general formula of Theorem~\ref{thm:main} yields
\begin{equation}
\label{eq:nodal2}
\vol(\calZ_f) = -\frac12
\int_M \div \left( F_1 \circ \frac{\grad f}{\abs{f}} \right) \vol_M
\end{equation}
(where minus the absolute value appears since $f$ is negative on $M^0_f$ and positive on $M^0_{-f}$).
Of course, this identity could have been obtained directly by applying the divergence theorem to the identity $\vol(\calZ_f) = \int_{\partial M^0_f} \vol_{\partial M^0_f}$.

Recalling the definition of $\eta_{f, \ell}$ by Equation~\eqref{eq:etaell}, we set
\begin{equation}
\label{eq:eta}
\eta_f \coloneqq \eta_{f, 1} = \sqrt{f^2 + \norm{\grad f}^2}.
\end{equation}
We also write $\sigma_f \colon M \to \{-1, 0, 1\}$ for the sign of $f$.

Setting, in Formula~\eqref{eq:nodal2}, $F_1(u) \coloneqq G_1(\norm{u}) u$ with respectively $G_1(x) \coloneqq (1+x^2)^{-1/2}$ and $\frac2\pi \frac{\arctan x}{x}$ and $\frac{\tanh x}{x}$, one obtains
\begin{equation}
\label{eq:algebraic}
\vol(\calZ_f) = \frac12 \int_M
\frac{\sigma_f}{\eta_f^3} \left(
f \norm{\grad f}^2 + \Hess(f)(\grad f, \grad f)
- \eta_f^2 \lapla f \right)
\vol_M
\end{equation}
and
\begin{multline}
\label{eq:arctan}
\vol(\calZ_f) = \frac1\pi \int_M \left( 
\norm{\grad f}^{-1} \left( \arctan \circ \frac{\norm{\grad f}}{f} \right)
\left( \frac{\Hess(f)(\grad f, \grad f)}{\norm{\grad f}^2} - \lapla f \right) + {}
\right.\\\left.
\eta_f^{-2} \left( \norm{\grad f}^2 - \frac{f \Hess(f)(\grad f, \grad f)}{\norm{\grad f}^2} \right)
\right)
\vol_M
\end{multline}
and
\begin{multline}
\label{eq:tanh}
\vol(\calZ_f) = \frac12 \int_M \left(
\norm{\grad f}^{-1} \left( \tanh \circ \frac{\norm{\grad f}}{f} \right)
\left( \frac{\Hess(f)(\grad f, \grad f)}{\norm{\grad f}^2} - \lapla f \right) + {}
\right.\\\left.
\left( \cosh \circ \frac{\norm{\grad f}}{f} \right)^{-2}
\left( \frac{\norm{\grad f}^2}{f^2} - \frac{\Hess(f)(\grad f, \grad f)}{f \norm{\grad f}^2} \right)
\right)
\vol_M
\end{multline}
(see~\cite{jubin} for the computation details).

\begin{remark}
\label{rmk:cont}
In the last three formulas, all terms of the integrands are bounded on $M$ and continuous on $M \setminus \calZ_f$.
Indeed, the Hessian expressions are quadratic in $\norm{\grad f}$, the $\arctan$ and $\tanh$ expressions are linear in $\norm{\grad f}$ when $\norm{\grad f}$ is small, and the $\cosh$ expression is exponentially small in $\abs{f}$ when $\abs{f}$ is small.
However, not all terms need be continuous on~$M$.
This problem is dealt with below.
\end{remark}

\begin{remark}
\label{rmk:corners2}
Fulfilling the promise made in Remark~\ref{rmk:corners}, let $M$ be a compact Riemannian manifold with boundary.
If $f$ intersects $\partial M$ transversely, then Formula~\eqref{eq:nodal2} becomes
\begin{equation}
\label{eq:corner}
\vol(\calZ_f) = \frac12 \left( \int_{\partial M} \innerp*{F \circ \frac{\grad f}{f}}{v} \vol_{\partial M} - \int_M \left( \div \left( F \circ \frac{\grad f}{f} \right) \right) \vol_M \right).
\end{equation}
Note that by the transversality assumption, $\calZ_f \cap \partial M$ is negligible in $\partial M$.
This formula reduces in dimension~1 to~\cite[Prop.~3]{angst}.
\end{remark}

\begin{remark}
The cases considered in~\cite{angst} correspond to $M = (\R/\Z)^n$ with the standard flat metric.
In particular, Formula~\eqref{eq:algebraic} is essentially~\cite[Prop.~5]{angst} (in the case of $(\R/\Z)^n$ with the standard flat metric).
In dimension~1, the general formula~\eqref{eq:nodal2} reduces to~\cite[Prop.~2]{angst}, and in that case, only the condition $\lim_{x \to \pm\infty} F(x) = \pm 1$ is required if one considers the integral as an improper Lebesgue integral.
Similarly, Formula~\eqref{eq:algebraic} reduces to~\cite[Prop.~1]{angst} and Formula~\eqref{eq:arctan} to~\cite[Cor.~1 of Prop.~2]{angst}.
\end{remark}

More generally, the intrinsic volumes of subsets have the additivity property
\begin{equation}
\calL_{n-k}(A) + \calL_{n-k}(B) = \calL_{n-k}(A \cup B) - \calL_{n-k}(A \cap B)
\end{equation}
when $A, B$ are subsets of a compact $n$-dimensional Riemannian manifold~$M$ such that all terms are well-defined (see~\cite[Thm.~5.16(6)]{federer}).
Therefore, if $f \in \calC^3_{0-\reg}(M, \R)$, then
\begin{equation}
\calL_{n-k}(\calZ_f) = \calL_{n-k}(M^0_f) + \calL_{n-k}(M^0_{-f}) - \calL_{n-k}(M).
\end{equation}
The terms corresponding to the first summand in~\eqref{eq:main} cancel out, so that
\begin{equation}
\calL_{n-k}(\calZ_f) = \sum_{m=0}^{\floor*{\frac{k-1}{2}}} b_{k, m} \int_M
\sigma_f^k \div \left( \frac{P_{n, k, m}(R, \Hess(f), \grad f)}{\eta_{f, 3k-2(m+1)}} \grad f \right) \vol_M.
\end{equation}
The exponent $k$ of $\sigma_f$ is congruent modulo 2 to $\deg_{\Hess(f)} P_{n, k, m} + \deg_{\grad f} P_{n, k, m} + 1 = 3k - 2(1+m)$ by Remark~\ref{rmk:degree}.
In particular, $\calL_{n-k}(\calZ_f) = 0$ for $k$ even, as expected.
One can also consider $\calZ_f$ as a Riemmannian manifold with curvature $\tilde{R}$ and obtain
\begin{equation}
\calL_{n-1-k}(\calZ_f) = a_k \int_{\calZ_f} \trace \left( \extpow{k/2} \tilde{R} \right) \vol_{\calZ_f}
\end{equation}
where $\tilde{R}$ is given by the Gauss formula for the curvature of submanifolds, $\tilde{R}(X, Y, Z, T) = R(X, Y, Z, T) + S(X, Z) S(Y, T) - S(X, T) S (Y, Z)$ for $X, Y, Z, T \in \frakX(\calZ_f)$.\\

We return to the question raised in Remark~\ref{rmk:cont} of having continuous integrands.
The only non-continuous terms in the integrands of Equations~\eqref{eq:algebraic}, \eqref{eq:arctan}, \eqref{eq:tanh} are of the form
\begin{equation}
\sigma_f h \left( \Hess(f)(\grad f, \grad f) - \lapla f \: \norm{\grad f}^2 \right)
\end{equation}
with $h \in \calC^1(M, \R)$, respectively $h = \eta_f^{-3}$ and $h = \norm{\grad f}^{-3} \left(\arctan \circ \frac{\norm{\grad f}}{\abs{f}} \right)$ and $h = \norm{\grad f}^{-3} \left(\tanh \circ \frac{\norm{\grad f}}{\abs{f}} \right)$.
This is dealt with in~\cite{angst} (in the case of Equation~\eqref{eq:algebraic} on a flat torus) using an integration by parts.
The same method extends to compact Riemannian manifolds as follows.
One has
\begin{equation*}
\Hess(f)(\grad f, \grad f) - (\lapla f) \norm{\grad f}^2 =
\innerp*{\grad f}{\levici_{\grad f} \grad f - (\lapla f)\grad f}.
\end{equation*}
Therefore,
\begin{equation*}
\sigma_f h \left( \Hess(f)(\grad f, \grad f) - (\lapla f) \norm{\grad f}^2 \right) =
\innerp*{\grad{\abs f}}{h \big( \levici_{\grad f} \grad f - (\lapla f)\grad f \big)}.
\end{equation*}
We temporarily assume that $f$ is of class $\calC^3$ and we use the fact that $\div \left( \abs{f} h (\levici_{\grad f} \grad f - (\lapla f)\grad f) \right)$ has a vanishing integral on $M$ (by the standard divergence theorem).
Therefore,
\begin{multline*}
\int_M \sigma_f h \left( \Hess(f)(\grad f, \grad f) - (\lapla f) \norm{\grad f}^2 \right) \vol_M =\\
\int_M \abs{f} \div \left( h \left( (\lapla f)\grad f - \levici_{\grad f} \grad f \right) \right) \vol_M.
\end{multline*}
One has $\div((\lapla f)\grad f) = (\lapla f)^2 + \innerp{\grad \lapla f}{\grad f}$.
The Bochner formula yields
\begin{align*}
\div \left( \levici_{\grad f} \grad f \right)
&= \div \left ( \frac12 \grad \norm{\grad f}^2 \right) \\
&= \frac12 \lapla \norm{\grad f}^2\\
&= \innerp*{\grad \lapla f}{\grad f} + \norm{\Hess f}^2 + \Ric(\grad f, \grad f)
\end{align*}
where the norm of the Hessian is the Hilbert--Schmidt norm.
Therefore, the third derivatives cancel out.
Since $\calC^2(M, \R)$ is dense in $\calC^3(M, \R)$ for the (Whitney) strong $\calC^2$-topology (see for instance~\cite[Thm.~II.2.6]{hirsch}, and the next section for function space topologies) and the involved quantities are continuous in this topology, one has, for any $f$ of class $\calC^2$,
\begin{multline}
\int_M \sigma_f h \left( \Hess(f)(\grad f, \grad f) - (\lapla f) \norm{\grad f}^2 \right) \vol_M =\\
\int_M \abs{f}
\left(
h \left(
(\lapla f)^2 - \norm{\Hess f}^2 - \Ric(\grad f, \grad f)
\right) +
\right.\\\left.
\innerp*{\grad h}{(\lapla f)\grad f - \levici_{\grad f} \grad f} \vphantom{(\lapla f)^2} \right) \vol_M.
\end{multline}

For example, one has $\grad{\eta_f^{-3}} = -3 \eta_f^{-5} \left( f \grad f + \levici_{\grad f} \grad f\right)$, so Formula~\eqref{eq:algebraic} becomes
\begin{multline}
\label{eq:lipschitz}
\vol(\calZ_f) = \frac12 \int_M
\left(
\frac{\abs f}{\eta_f^3} \left( \norm{\grad f}^2 - \abs f \lapla f + (\lapla f)^2 - \norm{\Hess f}^2 - \Ric(\grad f, \grad f) \right) + {}
\right.\\\left.
3 \eta_f^{-5} \left( \vphantom{(\lapla f)^2} f \Hess(f)(\grad f, \grad f) + \Hess(f)(\grad f, \levici_{\grad f} \grad f) - {}
\right.\right.\\\left.\left.
(\lapla f)^2 (f \norm{\grad f}^2 + \Hess(f)(\grad f, \grad f)) \right)
\vphantom{\frac{\abs f}{\eta_f^3}} \right)
\vol_M.
\end{multline}
In the case of a flat torus, this is~\cite[Prop.~7]{angst}.

\begin{remark}
These formulas can also be written in terms of the tracefree Hessian.
Recall that $\Hess^0(f) = \Hess(f) - \frac{\lapla f}{n} \id$.
A tracefree linear map is Hilbert--Schmidt-orthogonal to the identity, so $\norm{\Hess f}^2 = \frac{(\lapla f)^2}{n} + \norm{\Hess^0 f}^2$.
\end{remark}

In Equation~\eqref{eq:lipschitz}, the integrand is a Lipschitz continuous functional of $f \in \calC^2_{0-\reg}(M, \R)$ (see next section for the precise setting), so one can apply techniques of the Malliavin calculus (see~\cite{angst}).
The only difference between~\eqref{eq:lipschitz} and~\cite[Prop.~7]{angst} is the additional term involving the Ricci curvature, $\abs f \eta_f^{-3} \Ric(\grad f, \grad f)$, and this term is in the required domain of the Malliavin calculus by the same proof as~\cite[Lem.~2 p.~26]{angst}.
Therefore, \cite[Thm.~1]{angst} holds on any compact Riemannian manifold.
Similarly, Formula~\eqref{eq:corner} shows that the extra boundary terms are not problematic, so~\cite[Thm.~1]{angst} holds on any compact Riemannian manifold with corners, a generalization which includes~\cite[Thm.~2]{angst} as a special case.

\section{Continuity of the intrinsic volumes of sublevel sets}
\label{sec:continuity}

\subsection{Review of function space topologies}

For this subsection, we refer to~\cite[Ch.~II]{hirsch} for details.
Let $(M, g)$ be a Riemannian manifold and $p \in \overline{\N}$.
We will use two different topologies on the set $\calC^p(M, \R)$, the \define{uniform $\calC^p$-topology}, and the finer (Whitney) \define{strong $\calC^p$-topology}.
The resulting topological spaces will be denoted with the subscripts $U$ and $S$ respectively.
The first is a completely metrizable group and the second is a Baire topological group (countable intersections of dense open subsets are dense).
In particular, there is a notion of Lipschitz continuity (by which we mean ``local Lipschitz continuity''\footnote{Note that (local) Lipschitz continuity does not imply uniform continuity when the domain is not complete, as the sign function on $\R_{\neq 0}$ shows.}) for maps between the first space and other metric spaces.
The product $\calC^p(M, \R) \times \R$ will be considered with the corresponding product topology, and the sets $\Reg^p(M, \R)$ and $\Reg_c^p(M, \R)$ defined in Equations~\eqref{def:reg} and~\eqref{def:regc} with the corresponding subspace topologies.

If $0 \leq i \leq p$ and $f \in \calC^p(M, \R)$, then $\levici^i f \in \Gamma^{(p-i)}(\bigotimes^i TM \to M)$.
We denote by $\norm{\levici^i f(x)}$ the norm of this multilinear form induced by the norm~$g_x$ on~$T^*_xM$, and by $\norm{\levici^i f}_\infty$ the supremum of these norms for $x \in M$.
Let $p \in \N$.
For the uniform $\calC^p$-topology, a neighborhood basis of 0 is given by
\begin{equation}
U_p(\epsilon) \coloneqq \{ f \in \calC^p(M, \R) \mid \sum_{i=0}^p \norm{\levici^i f}_\infty < \epsilon \}
\end{equation}
for $\epsilon \in \R_{>0}$.
For the strong $\calC^p$-topology, a neighborhood basis of 0 is given by
\begin{equation}
S_p(\epsilon) \coloneqq \{ f \in \calC^p(M, \R) \mid \forall x \in M \; \sum_{i=0}^p \norm{\levici^i f(x)} < \epsilon(x) \}
\end{equation}
for $\epsilon \in \calC^0(M, \R_{>0})$.
The uniform and strong $\calC^{\infty}$-topologies are obtained as the unions of the corresponding $\calC^p$-topologies.

The strong $\calC^p$-topology does not depend on the Riemannian metric (it could actually be defined using norms of usual derivatives in charts).
These topologies differ in the control of functions at infinity (in particular, they are equal when $M$ is compact).
Results involving the strong topology will often remain true for the uniform topology when restricted to proper functions.

The strong topology has the disadvantage that the inclusion of constant functions, $\R \to \calC^p(M, \R)_S, a \mapsto (x \mapsto a)$, is not continuous.
For example, the function
\begin{equation}
\label{eq:tau}
\begin{aligned}
\tau \colon \calC^p(M, \R)_U \times \R &\longrightarrow \calC^p(M, \R)_U\\
(f, a) &\longmapsto f - a.
\end{aligned}
\end{equation}
is Lipschitz continuous, but the analogous result (for mere continuity) with the strong topology does not hold.
Therefore, when studying sublevel sets at varying heights, we will use the uniform topology and we will restrict our attention to proper functions, and when studying sublevel sets at a fixed height, we will use the strong topology if we want to allow nonproper functions.

We denote by $\calC^p_\p(M, \R)$ (resp.\ $\calC^p_\bb(M, \R)$) the set of proper (resp.\ bounded below) functions in $\calC^p(M, \R)$, and by $\calC^p_\pb(M, \R) \coloneqq \calC^p_\p(M, \R) \cap \calC^p_\bb(M, \R)$ the set of proper bounded below functions.
Similarly, we set $\Reg_*^p \coloneqq \Reg^p(M, \R) \cap (\calC_*^p(M, \R) \times \R)$ for $* = \bb, \p, \pb$.

\begin{proposition}
\label{prop:clopen}
Let $p \in \overline{\N}$.
At least one (resp.\ all) sublevel set(s) of $f \in \calC^p(M, \R)$ is/are compact if and only if $f$ is bounded below (resp.\ proper bounded below).
In particular, $\Reg_\pb^p(M, \R) \subseteq \Reg_c^p(M, \R)$.
The subsets $\calC^p_\bb(M, \R)$ and $\calC^p_\p(M, \R)$ and $\calC^p_\pb(M, \R)$ are open and closed in $\calC^p(M, \R)_U$.
\end{proposition}

\begin{proof}
Obvious.
\end{proof}

\begin{example}
For a given function, the set of real numbers such that the associated sublevel set is compact can be any downset.
Indeed, consider the functions on $\R$ which send $x$ to respectively $x$ or $a$ or $a + e^x$ or $x^2$.
The sets of real numbers such that the associated sublevel set is compact are $\varnothing$ and $\intervOO{-\infty}{a}$ and $\intervOC{-\infty}{a}$ and $\R$ respectively.
If in the second case one requires that~$a$ be a regular value, then consider $x \mapsto a - e^x$.
\end{example}

Recall that $\eta_f$ was defined by Equation~\eqref{eq:eta} and $\tau$ by Equation~\eqref{eq:tau}.

\begin{lemma}
The function $\inf \colon \calC^0(M, \R)_U \to \R \cup \{-\infty\}$ is Lipschitz continuous.
The function $\eta \colon \calC^1(M, \R)_U \to \calC^0(M, \R)_U$ is Lipschitz continuous.
Let $p \in \overline{\N}$.
The function $\eta \colon \calC^{1+p}_{0-\reg}(M, \R)_X \to \calC^p(M, \R_{>0})_X$ is Lipschitz continuous for $X = U$ and continuous for $X = S$.
The function $m \coloneqq \inf \circ \eta \circ \tau \colon \calC^{1+p}(M, \R)_U \times \R \to \R$ is Lipschitz continuous.
\end{lemma}

\begin{proof}
Obvious.
\end{proof}

\begin{proposition}
Let $p \in \overline{\N_{\geq 1}}$.
The subset $\Reg_\p^p(M, \R)$ is open and dense in $\calC_\p^p(M, \R)_U \times \R$.
The subset $\calC_{0-\reg}^p(M, \R)$ (resp.\ $\calC_{0-\reg, c}^p(M, \R)$) is open and dense in $\calC^p(M, \R)_S$ (resp.\ $\calC_{0-c}^p(M, \R)_S$).
The three openness results actually hold for the (uniform or strong) $\calC^0$-topology.
\end{proposition}

\begin{proof}
One has, $\Reg_\p^p(M, \R) = \left(m|_{\calC_\p^p(M, \R)_U}\right)^{-1}\left(\R_{>0}\right)$, which is therefore open in $\calC_\p^p(M, \R)_U$.
Similarly, $\calC_{0-\reg}^p(M, \R) = \eta^{-1}\left(\calC^{p-1}(M, \R_{>0})\right)$ is open in $\calC_\p^p(M, \R)_S$.

As for density, by the Morse--Sard theorem, if $f \in \calC^n(M, \R)$, then the set of regular values of $f$ is dense.
This implies that $\Reg_\p^p(M, \R)$ is dense in $\calC^{\max(n, p)}(M, \R)_U \times \R$, which is dense in $\calC^p(M, \R)_U \times \R$ (see~\cite[Thm.~II.2.6]{hirsch}).

For $\calC_{0-\reg}^p(M, \R)$, one can use the transversality theorem as follows.
Let $f \in \calC^p(M, \R)$ and $\epsilon \in \calC^1(M, \R_{>0})$.
Let $(\phi_i)_{i \in I}$ be a smooth partition of unity subordinated to some locally finite atlas of $M$.
Consider the map $\Phi \colon M \times \R^{\abs{I}} \to \R, (x, (\lambda_i)) \mapsto f(x) + \epsilon(x) \sum_i \lambda_i \phi_i(x)$.
Then, $\Phi$ is submersive, so for almost all tuples $(\lambda_i)$, the map $\Phi(-, (\lambda_i))$ is transverse to 0.

The proofs work similarly with the compactness requirement added.
\end{proof}

\subsection{Continuity of the intrinsic volumes}

We begin with the special cases of the volume and the nodal volume, which will be needed in the proof of the general case.
We actually prove a more general statement, where $\triangle$ denotes the symmetric difference of two sets.

\begin{proposition}
\label{prop:volume}
Let $(M, g)$ be a Riemannian manifold.
The functions
\begin{align}
\Reg_\pb^1(M, \R)_U^2 &\longrightarrow \R_{\geq 0}\\
\big( (f, a), (g, b) \big) &\longmapsto \vol(M^a_f \triangle M^b_g)\notag
\shortintertext{and}
\Reg_\pb^1(M, \R)_U &\longrightarrow \R_{\geq 0}\\
(f, a) &\longmapsto \vol(M^a_f)\notag
\end{align}
are continuous (resp.\ Lipschitz continuous) when the domains are given the uniform $\calC^0$ (resp.\ $\calC^1$)-topology.
The functions
\begin{align}
\calC_{0-\reg, c}^1(M, \R)_S^2 &\longrightarrow \R_{\geq 0}\\
(f, g) &\longmapsto \vol(M_f^0 \triangle M_g^0)\notag
\shortintertext{and}
\calC_{0-\reg, c}^1(M, \R)_S &\longrightarrow \R_{\geq 0}\\
f &\longmapsto \vol(M_f^0)\notag
\end{align}
are continuous when the domains are given the strong $\calC^0$-topology.

The function
\begin{align}
\Reg_\pb^1(M, \R)_U &\longrightarrow \R_{\geq 0}\\
(f, a) & \longmapsto \vol(\partial M^a_f)\notag
\end{align}
is continuous when the domain is given the uniform $\calC^1$-topology.
\end{proposition}

\begin{remark}
The continuity of nodal volumes was proved in the Euclidean case in~\cite[Thm.~3]{app}, with a similar proof.
\end{remark}

We first prove a lemma.

\begin{lemma}
\label{lem:neighb}
Let $f \in \calC^1_{0-\reg, c}(M, \R)$.
For any neighborhood $U$ of $\calZ_f$, there exists an open neighborhood $V$ of $f$ in the strong $\calC^0$-topology (and, if $f$ is proper, in the uniform $\calC^0$-topology) such that for any $h, k \in V$, one has $\calZ_h \subseteq U$ and $M^0_h \triangle M^0_k \subseteq U$.
\end{lemma}

\begin{proof}
Let $f$ and $U$ be as in the statement.
Then $M_f^0$ is compact.
Let $H, K$ be compact subsets such that
$M_0^f \setminus U \subseteq H \subset\subset M_0^f \subset\subset K \subseteq M_0^f \cup U$.
By Proposition~\ref{prop:levsublev}, $f$ is strictly negative on $\interior{M^0_f}$, so $\epsilon_0 \coloneqq \inf \{ -f(x) \mid x \in H \} > 0$.

In the proper case, one has $\epsilon_1 \coloneqq \inf \{f(x) \mid x \in M \setminus K \} > 0$ and we set $V \coloneqq f +  U_0(\min(\epsilon_0, \epsilon_1))$.
In the nonproper case, let $\epsilon_2 \coloneqq \min(\epsilon_0, \inf \{f(x) \mid x \in \partial K \}) > 0$, let $\epsilon \in \calC^0(M, \R_{>0})$ be the function equal to $\epsilon_2$ on $K$ and $\min(\epsilon_2, f)$ on $M \setminus K$, and set $V \coloneqq f + S_0(\epsilon)$.
In both cases, if $h \in V$, one has $\calZ_h \subseteq K \setminus H \subseteq U$, and similarly $M^0_h \triangle M^0_k \subseteq U$ if $h, k \in V$.
\end{proof}

\begin{proof}[Proof of the proposition]
(i)
Let $\Phi$ be the first function in the proposition.
Since $M_f^a = M^0_{\tau(f, a)}$ and $\tau$ is Lipschitz continuous, it suffices to consider 0-sublevel sets.
Let $((f, 0), (g, 0)) \in \Reg_\pb^1(M, \R)^2$.
If $((h, 0), (k, 0)) \in \Reg_\pb^1(M, \R)^2$, then
$(M^0_f \triangle M^0_g) \triangle (M^0_h \triangle M^0_k) \subseteq (M^0_f \triangle M^0_h) \cup (M^0_g \triangle M^0_k)$,
so
$\abs{\Phi((f, 0), (g, 0)) - \Phi((h, 0), (k, 0))} \leq
\Phi((f, 0), (h, 0)) + \Phi((g, 0), (k, 0))$.
Therefore, it suffices to prove the Lipschitz continuity of~$\Phi$ on the diagonal.
Continuity is proved by Lemma~\ref{lem:neighb} and we defer the proof of Lipschitz continuity to the end of the proof.

(ii)
In the non-proper case, it is similarly sufficient to prove the continuity of the third function, say $\Psi$, on the diagonal, and the Lemma~\ref{lem:neighb} also proves the result.

(iii)
The second (resp.\ fourth) function is equal to $\Phi(-, (0,-1))$ (resp.\ $\Psi(-, 1)$), so it is continuous.

(iv)
We now prove continuity of nodal volumes.
By Lipschitz continuity of $\tau$, we can restrict our attention to $(f, 0) \in \Reg_\pb^1(M, \R)_U$.
Then, $\calZ_f$ is a compact $\calC^1$-hypersurface.
For any $x \in \calZ_f$, there exists a smooth chart $\phi \colon U \to \R^n$ such that $\phi(U) = \intervOO{0}{1}^d$ and $\phi(\calZ_f \cap U)$ is the graph of a $\calC^1$-function $f^\phi \colon \intervOO{0}{1}^{d-1} \to \intervOO{1/3}{2/3}$ with $\norm{\extdiff f^\phi}_\infty < 1$.
Since $\calZ_f$ is compact, there exists a finite cover $(U_i)_{i \in I}$ of $\calZ_f$ by such sets.
Let $(\psi_i)_{i \in I}$ be a smooth partition of unity subordinated to $(U_i)_{i \in I}$.

By Lemma~\ref{lem:neighb}, there exists a neighborhood $V$ of $f$ such that the nodal set of any $h \in V$ is included in $\bigcup_{i \in I} U_i$.
For any $i \in I$, by the implicit function theorem, there exists a neighborhood $V_i \subseteq V$ of $f$ in the $\calC^1$-topology such that for any $h \in V_i$, the hypersurface $\phi_i(U_i \cap \calZ_h)$ is the graph of a function $h_i \colon \intervOO{0}{1}^{d-1} \to \intervOO{0}{1}$ with $\norm{\extdiff h_i}_\infty < 2$.
One then has $\grad h_i = -\left(\frac{\partial({\phi_i}_* h)}{\partial x_d}\right)^{-1} \grad ({\phi_i}_* h)$.
Set $W \coloneqq \bigcap_{i \in I} V_i$.

Let $h \in W$ and let $j \colon \calZ_g \hookrightarrow M$ be the inclusion.
Then,
\begin{align*}
\vol(\calZ_h)
&= \sum_{i \in I} \int_{U_i \cap \calZ_h} j^*(\psi_i \: \vol_M)\\
&= \sum_{i \in I} \int_{\intervOO{0}{1}^{d-1}} \sqrt{1+\norm{\grad h}^2} {\phi_i}_* j^*(\psi_i \: \vol_M)\\
&= \sum_{i \in I} \int_{\intervOO{0}{1}^{d-1}} \sqrt{1+\left(\frac{\partial({\phi_i}_* h)}{\partial x_d}\right)^{-2} \norm{\grad ({\phi_i}_* h)}^2} (\pr_{d-1} \circ \phi_i)_*(\psi_i \: \vol_M)
\end{align*}
where $\pr_{d-1} \colon \R^d \to \R^{d-1}$ is the projection on the first $d-1$ components.
On this expression, the continuity of $\vol(\calZ_-)$ for the uniform $\calC^1$-topology is clear.

(v)
We now prove the Lipschitz continuity statement.
Let $(f, 0) \in \Reg_\pb^1(M, \R)$.
By Lemma~\ref{lem:neighb}, there exists an open neighborhood $V$ of $f$ in the uniform $\calC^0$-topology such that 
Let $K$ be a compact neighborhood of $\calZ_f$ such that $\epsilon \coloneqq \frac12 \min \big(\inf \{ \abs{f(x)} \mid x \in M \setminus K \}, \inf \{ \norm{\grad f(x)} \mid x \in K \} \big) > 0$.
Let $h, k \in \calC_{0-\reg, c}^1(M, \R) \cap (f + U_1(\epsilon))$.
Then, $M_h^0 \triangle M_k^0 \subseteq K$.
Furthermore, $M_h^0 \triangle M_k^0$ is included in a tubular neighborhood of $\calZ_h$ of radius $\frac1{\epsilon}\norm{h-k}_\infty$.
Indeed, let $x \in M_h^0 \triangle M_k^0$.
One can suppose that $k(x) \leq 0 \leq h(x)$.
Therefore, $0 \leq h(x) \leq \norm{h-k}_\infty$.
Consider the maximal integral curve $c$ defined by $c(0) = x$ and $c'(t) = - \grad h(c(t))$.
Let $T \coloneqq \inf \{ t \geq 0 \mid c(t) \in \calZ_h \}$.
One has $c(\intervCC{0}{T}) \subseteq K$.
Since $\norm{\grad h(c(t)} \geq \norm{\grad f(c(t))} - \epsilon \geq \epsilon$, one has $d(c(0), c(T)) \leq \frac1\epsilon \norm{h-k}_\infty$.

Since by the previous part of the proof, the nodal volume is continuous, $\vol(\calZ_h)$ is bounded, say by $A > 0$, on a $\calC^1$-neighborhood $W$ of $f$.
Therefore, if $h, k \in \calC_{0-\reg, c}^1(M, \R) \cap (f + U_1(\epsilon)) \cap W$, then $\vol(M_h^0 \triangle M_k^0) \leq \frac{A'}{\epsilon} \norm{h-k}_\infty$ where $A' > 0$ only depends on $A$ and $(M, g)$.
\end{proof}

\begin{remark}
The volume function is not uniformly continuous, as the pairs of constant functions equal to $\pm \frac1n$ on a nonempty compact manifold~$M$ show: the volume of the 0-sublevel set jumps from 0 to $\vol(M)$ for two arbitrarily close functions.
Similar examples can be given for any nonzero intrinsic volume.
\end{remark}

\begin{remark}
The proof shows that the first two functions are actually pointwise Lipschitz continuous when the domain is given the uniform $\calC^0$-topology.
\end{remark}

\begin{theorem}
\label{thm:continuity}
Let $(M, g)$ be an $n$-dimensional Riemannian manifold.
If $0 \leq k \leq n$, then the function
\begin{alignat}{2}
\calL_{n-k}^\mathrm{sub} \colon&& \Reg_\pb^3(M, \R)_U &\longrightarrow \R\\
&&(f, a) & \longmapsto \calL_{n-k}(M^a_f)\notag
\shortintertext{is Lipschitz continuous, and the function}
\calL_{n-k}(M^0_-) \colon&& \calC^3_{0-\reg, c}(M, \R)_S &\longrightarrow \R\\
&&f &\longmapsto \calL_{n-k}(M^0_f)\notag
\end{alignat}
is continuous.
\end{theorem}

\begin{remark}
By Remark~\ref{rmk:C2}, the functions $\calL_{n-1}^\mathrm{sub}$ and $\calL_{n-1}(M^0_-)$ are also defined for $\calC^2$-functions, and the proof below applies.
For mere continuity in the proper case, the result even holds for $\calC^1$-functions, as proved in Proposition~\ref{prop:volume}.
\end{remark}

\begin{remark}
Since the Euler--Poincaré characteristic is an integer, one obtains that $\chi(M_-^-)$ is locally constant on $\Reg_\pb^3(M, \R)_U$ and $\chi(M_-^0)$ is locally constant on $\calC^3_{0-\reg, c}(M, \R)_S$.
\end{remark}

\begin{lemma}
Let $(M, g)$ be an $n$-dimensional Riemannian manifold and $p \in \overline{\N}$.
Let $n, k, m \in \N$.
The function
\begin{equation}
\begin{aligned}
Y_{n, k, m} \colon \Reg^{p+3}(M, \R)_U &\longrightarrow \calC^p(M, \R)_U\\
(f, a) & \longmapsto \div \left( \frac{P_{n, k, m}(R, \Hess(f), \grad f)}{\eta_{f-a, 3k-2(m+1)}} \grad f \right)
\end{aligned}
\end{equation}
is Lipschitz continuous.
The function $Y_{n, k, m}(-, 0) \colon \calC^{p+3}_{0-\reg}(M, \R)_S \to \calC^p(M, \R)_S$ is continuous.
\end{lemma}

\begin{proof}
The result follows from the continuity of the four functions
\begin{align*}
\tau \colon \Reg^{p+3}(M, \R)_U &\longrightarrow \calC^{p+3}_{0-\reg}(M, \R)_U,\\
\eta_{-, 3k-2(m+1)} \colon \calC^{p+3}_{0-\reg}(M, \R)_X &\longrightarrow \calC^{p+2}(M, \R_{>0})_X,\\
P_{n, k, m}(R, \Hess(-), \grad -) \grad - \colon \calC^{p+3}(M, \R)_X &\longrightarrow \frakX^{p+1}(M)_X,\\
\div \colon \frakX^{p+1}(M)_X &\longrightarrow \calC^p(M, \R)_X
\end{align*}
for $X = U, S$, with Lipschitz continuity when $X = U$.
\end{proof}

\begin{proof}[Proof of the theorem]
Because of the general identity $M^b_f = M^a_{f+a-b}$, it is sufficient to consider $(f, 0), (g, 0) \in \Reg_\p^3(M, \R)$.
By Formula~\eqref{eq:main}, one has
\begin{multline}
\abs*{\calL_{n-k}(M^0_g) - \calL_{n-k}(M^0_f)} \leq
\abs*{\int_{M^0_g} \trace\left( \extpow{k/2} R \right) \vol_M - \int_{M^0_f} \trace\left( \extpow{k/2} R \right) \vol_M} + {}\\
\sum_{m=0}^{\floor*{\frac{k-1}{2}}} \abs{b_{k, m}}
\abs*{\int_{M^0_g} Y_{n, k, m}(g, 0) \vol_M - \int_{M^0_f} Y_{n, k, m}(f, 0) \vol_M}.
\end{multline}
As for the first term, the integrand $\trace\left( \extpow{k/2} R \right)$ is locally bounded on $M$, and $\vol(M^0_g \triangle M^0_f)$ is controlled by $d(f, g)$ by Proposition~\ref{prop:volume}.

As for the summands, it suffices to bound $\int_{M^0_g \triangle M^0_f} \max(\abs*{Y_{n, k, m}(f, 0)}, \abs*{Y_{n, k, m}(g, 0)}) \vol_M$ and $\int_{M^0_f \cap M^0_g} \abs*{Y_{n, k, m}(g, 0) - Y_{n, k, m}(f, 0)} \vol_M$.
The first term is dealt with by Proposition~\ref{prop:volume} and the second by the lemma.
\end{proof}

\vspace*{5mm}
\noindent
\parbox[t]{21em}
{\scriptsize
Beno{\^i}t Jubin\\
Sorbonne Universit{\'e}\\
Institut de Math{\'e}matiques de Jussieu\\
F-75005 Paris France\\
e-mail: benoit.jubin@imj-prg.fr
}
\end{document}